\documentclass[11pt]{article}

\usepackage{amsmath,amsthm,amssymb,amsfonts,latexsym}
\usepackage{hyperref}
\usepackage[latin1]{inputenc}
\usepackage{longtable}
\usepackage{graphicx}
\usepackage{amsmath}
\usepackage{multirow}
\usepackage{setspace}
\usepackage{amssymb}
\usepackage{amsfonts}
\usepackage[font = small, labelfont=bf, justification=centering,skip=1pt]{caption}
\usepackage[font = small, labelfont=bf, justification=centering, skip=1pt]{subcaption}
\usepackage{geometry}
\usepackage{float}
\usepackage{color}
\usepackage{verbatim}
\usepackage{rotating}
\usepackage{pdflscape}
\usepackage[english]{babel}
\usepackage[latin1]{inputenc}
\usepackage[dvipsnames]{xcolor}
\usepackage{placeins}
\usepackage[
maxcitenames=2,
style=apa,
backend=biber,
sorting=anyt,
]{biblatex}
\usepackage{xcolor}
\usepackage[linesnumbered,ruled,vlined]{algorithm2e}
\usepackage{enumerate}
\usepackage{mathrsfs}
 
\bibliography{Bibliography}
\usepackage{amsthm}
\usepackage{graphicx}
\usepackage{graphicx}
\usepackage{epstopdf}
\usepackage{float} 
\usepackage{subcaption}
\usepackage{amsmath}
\usepackage{multicol}
\usepackage{hyperref}
\usepackage{textcomp}

\theoremstyle{definition}

\newtheorem{proposition}{Proposition}

\usepackage[latin1]{inputenc}
\usepackage{tikz}
\usetikzlibrary{shapes,arrows}
\tikzstyle{intermediate}=[circle,draw=black,fill=gray!20,thick,inner sep=0pt,minimum size=6mm]
\tikzstyle{od}=[circle,draw=black,fill=gray!70,thick, inner sep=0pt,minimum size=6mm]
\tikzstyle{intermediateaux}=[circle,draw=black,fill=gray!20,thick,inner sep=0pt,minimum size=9mm]
\tikzstyle{odaux}=[circle,draw=black,fill=gray!70,thick, inner sep=0pt,minimum size=9mm]
\usetikzlibrary{decorations}
\tikzstyle{decision} = [diamond, draw, fill=blue!20, 
    text width=6em, text badly centered, node distance=4cm, inner sep=0pt]
\tikzstyle{block} = [rectangle, draw, fill=blue!20, 
    text width=20em, text centered, rounded corners, minimum height=1em]
\tikzstyle{blockgreen} = [rectangle, draw, fill=red!20, 
    text width=20em, text centered, rounded corners, minimum height=1em]
\tikzstyle{blocksmall} = [rectangle, draw, fill=blue!20, 
    text width=8em, text centered, rounded corners, minimum height=1em]
\tikzstyle{line} = [draw, -latex']
\tikzstyle{cloud} = [draw, ellipse,fill=gray!20, node distance=5cm,
    minimum height=2em]
\usepackage{pgfplots}

\newcommand{\rev}[1]{\textcolor{black}{#1}}

\usepackage[normalem]{ulem} 

\usepackage{array}
\newcolumntype{H}{>{\setbox0=\hbox\bgroup}c<{\egroup}@{}}
\SetCommentSty{mycommfont}
\SetKwInput{KwInput}{Input}                
\SetKwInput{KwOutput}{Output}              
\SetKwBlock{DoParallel}{do in parallel}{end}

\usepackage{etoolbox}

\makeatletter
\patchcmd{\@algocf@start}
  {-1.5em}
  {0pt}
  {}{}
\makeatother

\SetArgSty{textnormal}

\begin{document}
\singlespacing
\numberwithin{equation}{section}

\title{\normalsize{13th AIMMS-MOPTA Optimization Modeling Competition}\\ 
\Large{An optimization and simulation method for the home service assignment, routing, and scheduling problem with stochastic travel time, service time, and cancellation}}

\author{Team \textit{The Optimistics} \\
Daniel Yam\'in, Daniel Barahona \\ Advisor: Alfaima L. Solano-Blanco}

\date{%
    \textit{Centro para la Optimizaci\'{o}n y Probabilidad Aplicada} (COPA), Departamento de Ingenier\'{\i}a Industrial\unskip, Universidad de los Andes, Bogot\'a, Colombia \\ %
    E-mail: \{\href{mailto:d.yamin@uniandes.edu.co}{\texttt{d.yamin}}, \href{mailto:ds.barahona@uniandes.edu.co}{\texttt{ds.barahona}}, \href{mailto:al.solano@uniandes.edu.co}{\texttt{al.solano}}\}@uniandes.edu.co
    %
    }


\maketitle

\begin{abstract}
\singlespacing\noindent
In the Home Service Assignment, Routing, and Appointment scheduling (H-SARA) problem, a set of homogeneous service teams must visit a set of customers. The home service provider needs to decide how many teams to hire (i.e., sizing problem), how to assign service teams to customers (i.e., assignment problem), how to route service teams (i.e., vehicle routing problem), and how to schedule the appointment times for the customers (i.e., appointment scheduling problem) such that the total cost is minimized. To tackle the H-SARA problem, we propose an efficient solution method that comprises two stages. In the first stage, we present a column generation algorithm to solve the sizing, assignment, and routing problem. The algorithm is enhanced by a high-quality initial solution which is found using the route-first cluster-second principle and a polynomial-time 2-approximation algorithm. In the second stage, due to the stochastic nature of travel time, service time, and cancellation, we propose a simulation-driven approach to decide the appointment times such that a desired on-time arrival probability is achieved. To ensure the suitability of the simulation model, we discuss the characterization of the stochastic parameters. The proposed ideas can be embedded in different solution schemes, including a fast heuristic method that finds good solutions within seconds or a more elaborate algorithm to find near-optimal solutions at the expense of longer computational time. At last, we provide a high-level flexible decision support tool implemented in AIMMS.
\noindent 
\end{abstract}

\bigskip 


\vfill

\thispagestyle{empty}

\pagebreak

\setcounter{page}{1}

\singlespacing




\section{Introduction}
Home services provide essential needs such as health care, beauty care, and banking services at the customers' homes. Due to several factors, including population aging, work obligations, and the outspread of chronic and infectious diseases, the demand for home services is expected to increase rapidly in the near future. In 2016, there were about 65,600 regulated, long-term care services providers in the United States which served more than 8.3 million people (\cite{harris2016long}). In most European countries, between 1\% and 5\% of the public budget is assigned to home health care services (\cite{genet2012home}). In a broader perspective, the worth of the global home service industry was estimated at \$282 billion and is expected to reach \$1,133.4 billion by 2026 (\cite{MarketResearch}). In a competitive market, lowering public expenditures, increasing service quality, and decreasing operational costs becomes a critical activity for home services providers (\cite{fikar2017home}). In light of the discussion above, the development of highly efficient computational models to support decision-making in the home service industry has become increasingly important.

Home services require professional service teams to travel between geographically distributed customers. Further, each served customer has an appointment time, referred as to a planned service start time. Therefore, in their day-to-day planning, home service providers need to address the following critical operational decisions: how many teams to hire (i.e., sizing problem), how to assign service teams to customers (i.e., assignment problem), how to route service teams (i.e., vehicle routing problem), and how to schedule the appointment times for the customers (i.e., appointment scheduling problem) such that their total operational cost is minimized. The above problem can be framed as the \emph{Home Service Assignment, Routing, and Appointment scheduling} (H-SARA) problem. Three random factors are considered to capture the stochastic nature of real world: travel time, service duration, and customer cancellation. Regardless of the high complexity that arises from various challenging optimization problems, the planning of home services is often performed manually (\cite{eveborn2006laps}). 

The H-SARA problem --and in a more general sense, the routing and scheduling of home care services-- is a major stream of research with different objectives and particular considerations (\cite{fikar2017home}). An overview of logistics management problems in the field of home services can be found in \textcite{gutierrez2013home}. From the operations research (OR) perspective, \textcite{milburn2012operations} describes tactical and operational planning problems arising in home health care. Also, and considering that the H-SARA problem is closely related to the vehicle routing problem (VRP), \textcite{toth2014vehicle} present many variants of the VRP and the most efficient solution methods to tackle them. Concretely in the routing and scheduling of home services, most solution methods rely on metaheuristics (\cite{akjiratikarl2007pso,mankowska2014home,braekers2016bi}), heuristics and approximation algorithms (\cite{eveborn2006laps,hindle2009travel}), and few exact optimization techniques (\cite{rasmussen2012home}). Despite that real-world operations are subject to uncertainty, most articles consider static information. 

\rev{More recently, the home service problem with underlying uncertainty has gained the attention of researchers. For instance, \textcite{chen2017tackling} tackle the problem by formulating an integer program with chance constraints to cope with uncertainty in durations. Also, \textcite{cappanera2021addressing} address the consistency and demand uncertainty in the home care planning problem. We refer to \textcite{zhan2021home} for a complete literature review on the home service routing and appointment scheduling with stochastic times. We highlight that most solution methods are either very complex or unable to find rapidly high-quality solutions for large-scale, real-world applications.}

In this study, we tackle the H-SARA problem by proposing an end-to-end methodology that solves the problem in a reasonable time and considers the stochastic nature of travel time, service time, and cancellation. The solution scheme comprises a column generation-based heuristic for the sizing, assignment, and routing (SAR) problem and a simulation model for the scheduling problem. To find an initial solution in the column generation scheme, we extend the route-first cluster-second split mechanism proposed by \textcite{prins2004simple}. To better control the simulation, we provide a framework to characterize stochastic travel time, service time, and cancellation. Finally, we embed the solution method into an AIMMS-based user-friendly application that serves as a high-level decision-support tool for home service providers.

\section{Problem definition} \label{s: problem-definition}

Mathematically, we can formulate the H-SARA problem as follows. Let $\mathcal{G} = (\mathcal{V},\mathcal{A})$ be a directed graph in which $\mathcal{V} = \{0,1,\ldots, n, n+1\}$ is the set of nodes, $\mathcal{N} = \{1,\ldots, n\} \subseteq \mathcal{V}$ is the set customers that needs to be serviced, and $\mathcal{A} = \{(i,j)\mid (i=0 \wedge j \in \mathcal{N}) \vee (i \in \mathcal{N} \wedge j \in \mathcal{N} \wedge i \ne j) \vee (i \in \mathcal{N} \wedge j = n+1)\}$ is the set of arcs. The depot is represented by two nodes, namely, $0$ and $n+1$. Each customer $i\in \mathcal{N}$ has a random service time $s_{i}$ and may cancel his/her appointment on the day of service. \rev{We denote by $p_{i}$ the probability associated with customer $i\in \mathcal{N}$ canceling its appointment.} Also, each arc $(i,j) \in \mathcal{A}$ has a random travel time $t_{ij}$. 

To service customers there is a set $\mathcal{M}$ of homogeneous service teams. We assume that $\mathcal{M}$ is a sufficiently large set since home service providers usually can hire third-party services. Each hired team $k \in \mathcal{M}$ departs from the depot (node $0$), services some customers, and returns to the depot (node $n+1$). Moreover, each customer is scheduled a service start time and the hired teams leave the depot at time $0$ and must return before time $L$ (\emph{a priori}). Given the stochastic nature of travel time, service time, and cancellations, we consider the following scenarios. First, the service team arrives at the customer's location before the scheduled service time and must wait (i.e., they remain idle until the scheduled start time). Second, the service team arrives at the customer's location after the scheduled start time and therefore the customer waits. Third, a service team needs additional time to finish serving all scheduled customers beyond $L$ and thus incurs in overtime. For the purpose of the present study, we assume that travel times satisfy the triangle inequality. Also, the probability distributions of travel time, service time, and cancellations are known.

There are several costs associated with the day-to-day activities of home service providers. First, a fixed cost $c_{f}$ is charged for each hired team. Second, $c_{t}$ and $c_{o}$ are the costs that the home service provider pays for one unit of travel time and for one unit of overtime, respectively. Finally, $c_{e}$ and $c_{d}$ can be seen as the cost that the company pays for being one time unit early and for being one time unit late while visiting a customer, respectively. These costs penalize early and late services.

We define a mathematical formulation for the H-SARA along the lines of  \textcite{toth2014vehicle, tacs2013vehicle}. Since travel and service times are random variables, we use their expected value. Henceforth, $t_{ij}$ denotes the expected travel time along arc $(i,j) \in \mathcal{A}$, whereas $s_{i}$ denotes the expected service time when visiting customer $i \in \mathcal{N}$. Note that using expected values is convenient by the linearity of expectation. Note too that the expected utility is the most straightforward metric to evaluate decision-making under uncertainty (\cite{li2016finding}). \rev{Alternatively, to add robustness to the optimization model, one could calculate a high percentile, say $95$, of the distribution of travel and service times, and use those values rather than the expected value.} Regarding variables, for each arc $(i,j)\in \mathcal{A}$ and each team $k \in \mathcal{M}$, let $x_{ij}^{k}$ be the binary flow variable indicating whether service team $k$ travels directly from node $i$ to node $j$ or not. To define the scheduling of the appointments, for each node $i\in \mathcal{V}$ and each team $k \in \mathcal{M}$, let $w_{i}^{k}$ be the expected time at which team $k$ begins service at node $i$. To incorporate the overtime cost, we define an auxiliary variable $\Delta^{k}$ for each team $k\in \mathcal{M}$. These variables capture the working time beyond time horizon $L$. For a vehicle $k\in \mathcal{M}$, given the values of $\boldsymbol{x}^{k} = \{x_{ij}^{k}\mid (i,j) \in \mathcal{A}\}$ and $\boldsymbol{w}^{k} = \{w_{i}^{k}\mid i \in \mathcal{N}\}$, we can compute the total expected earliness $E(\boldsymbol{x}^{k},\boldsymbol{w}^{k})$ and the total expected delay $D(\boldsymbol{x}^{k},\boldsymbol{w}^{k})$. To capture the expected travel time cost and the fixed hiring cost, for each arc $(i,j) \in \mathcal{A}$ we define
\begin{align*}
    c_{ij} = & \begin{cases}
    c_{t}\, t_{ij} + c_{f}, & i = 0; \\
    c_{t}\, t_{ij}, & i \in \mathcal{N}. 
    \end{cases}
\end{align*}
Finally, to model when a service team is idle (not being used), we add a fictitious arc between node 0 and node $n+1$ with a cost of zero. The mathematical program \eqref{1}--\eqref{9} describes the H-SARA problem.

\begin{align}
        \min\quad \sum_{k\in \mathcal{M}} \sum_{(i,j)\in \mathcal{A}} c_{ij}x_{ij}^{k} + \sum_{k \in \mathcal{M}} c_{o}\Delta^{k} + \sum_{k \in \mathcal{M}} c_{e}E(\boldsymbol{x}^{k},\boldsymbol{w}^{k}) + \sum_{k \in \mathcal{M}} c_{d} D(\boldsymbol{x}^{k},\boldsymbol{w}^{k}) \label{1}
\end{align}
\vspace{-1.2 cm}

\begin{align}
    \text{s.t., \,\,\,\, \quad \quad \quad \quad \quad \quad \quad \quad \quad}& \nonumber\\
    \sum_{k \in \mathcal{M}} \sum_{j \in \mathcal{N} \cup \{n+1\}}x_{ij}^{k} & = 1, & &    \forall i \in \mathcal{N} ;\label{2}\\
    \sum_{j \in \mathcal{N} \cup \{n+1\}}x_{0,j}^{k} & = 1, & & \forall k \in \mathcal{M}; \label{3} \\
    \sum_{i \in \mathcal{N} \cup \{0\}}x_{ij}^{k} - \sum_{i\in \mathcal{N} \cup \{n+1\}} x_{ji}^{k} & = 0, & & \forall k \in \mathcal{M}, j \in \mathcal{N}; \label{4}\\
    \sum_{i \in \mathcal{N} \cup \{0\}}x_{i,n+1}^{k} & = 1, & & \forall k \in \mathcal{M}; \label{5}\\
    w_{i}^{k} + (1-p_{i}) \times (s_{i} + t_{ij}) - w_{j}^{k} & \leq (1-x_{ij}^{k})|\mathcal{M}|, & & \forall k\in \mathcal{M}, (i,j)\in \mathcal{A} ;\label{6}\\
    w_{n+1}^{k} & \leq L + \Delta^{k}, & & \forall k \in \mathcal{M} ;\label{7}\\ 
    w_{i}^{k} & \geq 0, & & \forall k\in \mathcal{M}, i \in \mathcal{V}; \label{8}\\
    \Delta^{k} & \geq 0, & & \forall k\in \mathcal{M} ;\label{8.1}\\
    x_{ij}^{k}& \in \{0,1\}, & & \forall k \in \mathcal{M}, (i,j) \in \mathcal{A}.\label{9}
\end{align}

The objective function \eqref{1} minimizes the total cost, including the hiring cost and the expected travel time, overtime, earliness, and delay costs. Constraints \eqref{2} impose that each customer is served exactly once. Constraints \eqref{3}--\eqref{5} ensure that the route of each service team starts from the origin and ends at the destination (i.e., characterize a multi-commodity flow structure). \rev{Constraints \eqref{6} guarantee the consistency of the time variables. Note that constraints \eqref{6} consider customers cancellations and appointments are scheduled accordingly (see \S \ref{ss:simulation-model} for details)}. Soft constraints \eqref{7} ensure that the service teams return to the depot before $L$, or else they incur in overtime. Finally, constraints \eqref{8}, \eqref{8.1}, and \eqref{9} impose the non-negativity and binary requirements.

\section{Solution method} \label{s: solution-method}

The mathematical formulation \eqref{1}--\eqref{9} captures the key elements of the H-SARA problem. Given that the problem combines several difficult problems --and the routing problem alone is NP-Hard (\cite{garey1979computers})--, we establish that the H-SARA is also a NP-Hard problem, making it difficult to solve exactly. In this light, we decompose the problem to develop a two-stage solution method. Figure \ref{fig:diagram} presents an overview of our solution scheme. In the first stage, we use a column generation algorithm for the SAR problem. The algorithm builds from a set covering model and a high-quality initial solution which is found by extending the route-first cluster-second principle to the H-SARA problem. Subsequently, we iteratively solve a subproblem to find new promising routes for the service teams until the optimality conditions are met. In the second stage, we solve the appointment scheduling problem using a simulation-driven approach which is able to ensure the reliability of the solution. To guarantee the suitability of the simulation approach, we discuss the characterization of the stochastic parameters.

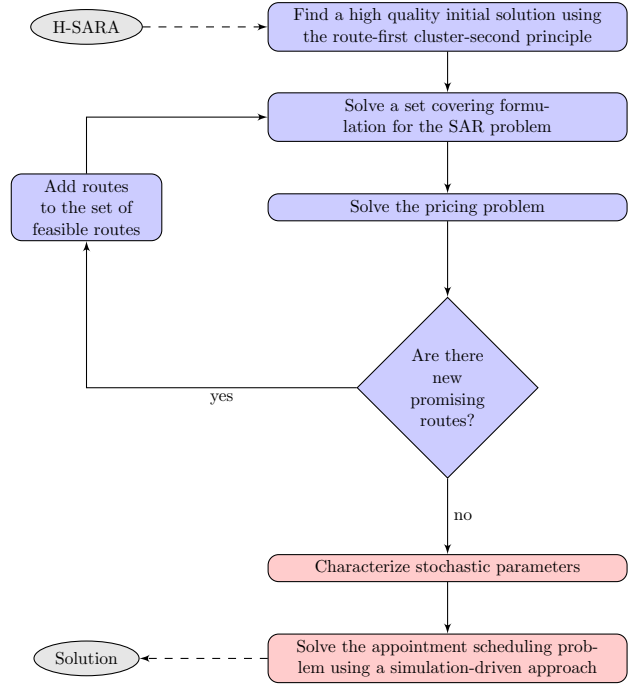
\begin{figure}[ht]
    \centering
    \begin{tikzpicture}[node distance = 2cm, auto, scale = 0.60, every node/.style={scale=0.60}]
        \node [block] (init) {Find a high quality initial solution using the route-first cluster-second principle};
        \node [cloud, left of=init, node distance=8cm] (expert) {H-SARA};
        \node [block, below of=init] (identify) {Solve a set covering formulation for the SAR problem};
        \node [block, below of=identify] (evaluate) {Solve the pricing problem};
        \node [blocksmall, left of=evaluate, node distance=8cm] (update) {Add routes to the set of feasible routes};
        \node [decision, below of=evaluate] (decide) {Are there new promising routes?};
        \node [blockgreen, below of=decide, node distance=4cm] (stop0) {Characterize stochastic parameters};
        \node [blockgreen, below of=stop0] (stop) {Solve the appointment scheduling problem using a simulation-driven approach};
        \node [cloud, left of=stop, node distance=8cm] (exit) {Solution};
        \path [line] (init) -- (identify);
        \path [line] (identify) -- (evaluate);
        \path [line] (evaluate) -- (decide);
        \path [line] (decide) -| node [near start] {yes} (update);
        \path [line] (update) |- (identify);
        \path [line] (decide) -- node {no}(stop0);
        \path [line] (stop0) -- (stop);
        \path [line,dashed] (expert) -- (init);
        \path [line,dashed] (stop) -- (exit);
    \end{tikzpicture}
    \vspace*{5mm}
    \caption{Solution method for the H-SARA problem with the first stage highlighted in blue and the second stage highlighted in red}
    \label{fig:diagram}
\end{figure}

The remainder of this paper is organized as follows. Section \ref{s: column-generation} explains the column generation algorithm for the SAR problem. Section \ref{s: appointment} presents the simulation-based method for the appointment scheduling problem. Section \ref{s: computational-experiments} shows the computational experiments. Finally, Section \ref{s: conclusion} concludes the paper and outlines future work. For the sake of the readability of the paper, we summarize the notation of the problem statement and solution method in Table \ref{NotationTable}.

\section{Column generation algorithm for the sizing, assignment, and routing problem} \label{s: column-generation}
In this section, we tackle the sizing, assignment, and routing problem using a column generation algorithm. In particular, with a column generation-based heuristic, we find the number of teams to hire, the assignment of customers to teams, and the routes that each team follows such that each customer is served and the fixed hiring cost, the expected travel time cost, and the expected overtime cost are minimized. We leave the appointment scheduling problem --and its associated costs of earliness and delay-- for Section \ref{s: appointment}. 

To build the column generation scheme, we first define a set covering formulation to build a column generation algorithm (\S \ref{ss:set covering}). Subsequently, we propose an efficient method to rapidly find a high-quality initial feasible solution for the problem mentioned above by applying the route-first cluster-second principle (\S \ref{ss: initial-solution}). Next, we present a mathematical programming-based solution method as well as an heuristic approach to solve the corresponding \emph{pricing problem} (\S \ref{ss:pricing}). Finally, we embed all these ideas into a column generation-based heuristic that solves the SAR problem in practical time for real-world applications (\S \ref{ss: cg-heuristic}).

\subsection{Set covering formulation and column generation} \label{ss:set covering}
By inspecting the mathematical formulation of the H-SARA problem, we see that the set of constraints \eqref{3}--\eqref{9} as well as the objective function \eqref{1} are separable for each service team $k\in \mathcal{M}$. By applying Dantzig-Wolfe decomposition, we can derive an equivalent set covering formulation of the problem (\cite{feillet2010tutorial}). 

Let $\Omega$ be the set of feasible routes satisfying constraints \eqref{3}--\eqref{5}. Note that each route represents a corresponding service team; that is, routes and service teams are equivalent. For each route $r \in \Omega$, let $c_{r}$ be the cost of the route (including the hiring, travel time, and overtime costs), $a_{ir}$ be a binary parameter indicating whether the route visits node $i \in \mathcal{N}$ or not, and $b_{ijr}$ be a binary parameter indicating whether the route uses arc $(i,j)\in \mathcal{A}$ or not. Finally, let $y_{r}$ be a binary variable which indicates if route $r \in \Omega$ is used in the solution or not. The SAR problem can be described by the integer program \eqref{11}--\eqref{13}.

\begin{align}
    \min\quad \sum_{r \in \Omega} c_{r} y_{r} & \label{11}\\
    \text{s.t.,\,\,\,\quad \quad \quad \quad}& \nonumber\\
    \sum_{r \in \Omega}a_{ir}y_{r} & \geq 1, & &    \forall i \in \mathcal{N}; \label{12}\\
    y_{r} & \in \mathbb{Z}_{+}^{1}, & & \forall r \in \Omega. \label{13}
\end{align}

The objective \eqref{11} minimizes the sum of the hiring cost, the expected travel time cost, and the expected overtime cost. Constraints \eqref{12} correspond to constraints \eqref{2} and guarantee that each customer is visited by a route. Finally, constraints \eqref{13} impose the integer nature of the variables. Unfortunately, the size of the set $\Omega$ grows exponentially with the number of nodes, making the set covering model intractable with a standard branch-and-bound algorithm. Thus, we tackle the problem with a column generation scheme. 

A lower (i.e., dual) bound on the optimal objective of the set covering model \eqref{11}--\eqref{13} can be obtained by dropping the integrality requirements \eqref{13} and replacing them with non-negative constraints. Then, the resulting linear programming (LP) relaxation is iteratively solved considering only a subset $\Omega \text{\textquotesingle} \subseteq \Omega$ of routes. For each covering constraint in \eqref{12} we have a corresponding dual variable $\pi_i \geq 0$ associated with customer $i\in \mathcal{N}$. Keeping a set partitioning formulation for constraints \eqref{12} --rather than the set covering formulation we picked-- would lead to free dual variables $\pi_i$, which typically slows down the convergence of the column generation algorithm (\cite{feillet2010tutorial}). At every iteration, a promising route --that is, one with negative \emph{reduced cost}-- can be found by solving a pricing problem as established by Proposition \ref{prop: pricing}.

\begin{proposition} 
For a given route $r\in \Omega$, the pricing problem is defined by the following integer program:

\begin{align}
    \min \quad \sum_{(i,j)\in \mathcal{A}}r_{ij}\,b_{ijr} +c_{o}\Delta^{r}  \label{FO_pricing} \\
    \text{s.t.,\,\;\, \quad\quad\quad\quad\quad\quad \quad \quad \quad}& \nonumber\\
    \sum_{j\mid (i,j)\in \mathcal{A}}b_{ijr} -\sum_{j\mid (j,i)\in \mathcal{A}}b_{jir} & = \begin{cases}
    1, & i = 0\\
    -1, & i = n+1 \\
    0, & \text{else}
    \end{cases}, & &    \forall i \in \mathcal{V} ; \label{balance}\\
    \sum_{j\mid (i,j)\in \mathcal{A}}b_{ijr} & \leq 1, & & \forall i \in \mathcal{V}; \label{capacity}\\
    \sum_{(i,j)\in \mathcal{A}} \tilde{t}_{ij}\, b_{ijr} &  \leq L + \Delta^{r}; \label{timelimit}\\
    \Delta^{r} & \geq 0; & &  \label{nature-pricing-delta}\\
    b_{ijr} & \in \{0,1\}, & & \forall (i,j)\in \mathcal{A}, \label{nature-pricing}
\end{align}
where $r_{ij}$ is defined as
\begin{align}
r_{ij} = \begin{cases}
    c_{f} + c_{t}\,t_{ij}, & i = 0;\\
    c_{t}\,t_{ij}-\pi_{i}, & i \in \mathcal{N},
    \end{cases} \label{reducedCost}
\end{align} 
and $\tilde{t}_{ij}$ is given by
\begin{align}
\tilde{t}_{ij} = \begin{cases}
    t_{ij}, & i = 0;\\
    t_{ij} + s_{i}, & i \in \mathcal{N},
    \end{cases} \label{time}
\end{align}

\noindent for ease of notation. Objective function \eqref{FO_pricing} minimizes the reduced cost. The \emph{mass balance constraints} \eqref{balance} guarantee that the outflow minus the inflow equals the supply/demand of the node. Constraints \eqref{capacity} ensure that the outgoing degree of each node is at most one. Constraints \eqref{timelimit} impose that the total expected time of the route (including travel times and service times) do not exceed the time limit $L$, or if it does, then $\Delta^{r}$ captures the overtime. Finally, constraints \eqref{nature-pricing-delta} and \eqref{nature-pricing} define the nature of the decision variables.

\label{prop: pricing}
\end{proposition}

\begin{proof}
Each feasible route must satisfy constraints \eqref{3}--\eqref{9}. Further, the reduced cost of a route $r$, denoted by $\overline{c}_{r}$, can be calculated as
\begin{align}
        \overline{c}_{r} & = c_{r} - \sum_{i \in \mathcal{N}}\pi_i\, a_{ir} \nonumber\\
        & = c_{f} + \sum_{(i,j)\in \mathcal{A}} c_{t}\, t_{ij}\, b_{ijr} + c_{o}\Delta^{r} - \sum_{i \in \mathcal{N}}\pi_i \,a_{ir} \label{auxiliar-reduced-cost}
\end{align}
Equality \eqref{auxiliar-reduced-cost} holds since we are only considering the hiring cost, the expected travel time cost, and the expected overtime cost. Since any feasible route starts at $0$, we can rewrite $\overline{c}_{r}$ as
\begin{align*}
        \overline{c}_{r} & = \sum_{j\mid (0,j)\in \mathcal{A}}c_{f}\,b_{ijr} + \sum_{(i,j)\in \mathcal{A}} c_{t}\, t_{ij}\, b_{ijr} + c_{o}\Delta^{r} - \sum_{i \in \mathcal{N}}\pi_i \,a_{ir} \\
        & = \sum_{(i,j) \in \mathcal{A} \mid i = 0}(c_{f} + c_{t}\, t_{ij})b_{ijr} + \sum_{(i,j)\in \mathcal{A} \mid i \in \mathcal{N}}(c_{t}\, t_{ij} - \pi_{i})b_{ijr} + c_{o}\Delta^{r}
\end{align*}
The above equation can be summarized by equation \eqref{FO_pricing}, where the weight for each arc is given by $r_{ij}$ as defined by equation \eqref{reducedCost}. 

\end{proof}

The problem defined by objective \eqref{FO_pricing} and constraints \eqref{balance}--\eqref{nature-pricing} is an \emph{Elementary Shortest Path Problem} with an additional soft time limit constraint.

\subsection{High-quality initial solution} \label{ss: initial-solution}
To start the column generation algorithm, we need an initial set of routes to derive an initial set of columns. On the one hand, having a high-quality initial solution usually accelerates the convergence of the column generation algorithm. On the other hand, the procedure to find the initial solution should be computationally fast. Therefore, in this study, we adapt the route-first cluster-second principle proposed by \textcite{prins2004simple} to the H-SARA problem. The route-first cluster-second split procedure for routing problems comprises two steps: (i) constructs a \emph{giant tour} --that is, solve a Traveling Salesman Problem (TSP)--; and (ii) split the giant tour into feasible trips.

We first need to find a Hamiltonian cycle, referring to a tour that visits all nodes in $\mathcal{N}$ exactly once. Unfortunately, finding such cycle is a NP-Complete problem and therefore difficult to solve exactly (\cite{garey1979computers}). Consequently, we use a polynomial-time 2-approximation algorithm to find such a cycle. Algorithm \ref{TSP-Approx} presents an overview of the approximation algorithm for the TSP (\cite{cormen2009introduction}). Given an undirected graph $\mathcal{G}^{t}$, the algorithm computes the minimum spanning tree $\mathcal{T}$ from root node $r$ using Prim's polynomial-time algorithm. Then, the resulting tree $\mathcal{T}$ is traversed in preorder and the nodes are listed in $\mathcal{H}$ according to when they are first visited. Finally, the algorithm returns $\mathcal{H}$, which is a Hamiltonian cycle. Proposition \ref{prop-TSP-Approx} establishes that Algorithm \ref{TSP-Approx} has a polynomial running time in the input size and also has an approximation ratio of 2 meaning that the cost of the solution returned is within a factor of 2 of the cost of an optimal tour.

{\singlespacing
\centering
\begin{algorithm}[H]
\footnotesize
\DontPrintSemicolon
  \KwInput{$\mathcal{G}^{t} = (\mathcal{V}^{t},\mathcal{A}^{t})$, undirected graph; $r$, root node.}
  \KwOutput{$\mathcal{H}$, Hamiltonian cycle.} 
  \text{compute a minimum spanning tree $\mathcal{T}$ for $\mathcal{G}^{t}$ from root node $r$ using Prim's algorithm}\\
  \text{define $\mathcal{H}$ as the list of vertices ordered according to when they are first visited in a preorder tree walk of $\mathcal{T}$}
  \Return{$\mathcal{H}$}
  \caption{Approximation algorithm for the TSP} 
  \label{TSP-Approx}
\end{algorithm}}

\begin{proposition} 
Algorithm \ref{TSP-Approx} is a polynomial-time 2-approximation algorithm for the TSP with the triangle inequality.
\label{prop-TSP-Approx}
\end{proposition}
\begin{proof}
See Theorem 35.2 in \textcite{cormen2009introduction}.
\end{proof}

We make the following two observations. First, the approximation algorithm's input is an \emph{undirected} graph $\mathcal{G}^{t} = (\mathcal{V}^{t},\mathcal{A}^{t})$ obtained from the original directed graph $\mathcal{G} = (\mathcal{V}, \mathcal{A})$. Specifically, we define $\mathcal{V}^{t} = \{0, \ldots, n\}$ and $\mathcal{A}^{t} = \{(i,j) \mid (i,j)\in \mathcal{A} \vee (j,i)\in \mathcal{A}\}$. We also define the cost function $c': \mathcal{A}^{t} \mapsto \mathbb{R}$ satisfying the triangle inequality as $c'_{ij} = c_{t}\,t_{ij}$ for each edge $(i,j) \in \mathcal{A}^{t}$. This means that to construct the giant tour, we only consider the expected travel time cost. Second, by using a polynomial-time 2-approximation algorithm, we are rapidly finding a Hamiltonian cycle whilst ensuring the quality of the solution found. 

Given a giant tour $\mathcal{H}$, we split it to find a set of routes that compose an initial feasible solution for the SAR problem. This procedure is referred to as split and uses an auxiliary directed acyclic graph of possible trips, denoted as $\mathcal{G}^{s} = (\mathcal{V}^{s}, \mathcal{A}^{s})$. Each arc $(i,j) \in \mathcal{A}^{s}$ represents the route that departs from the depot $0$, visits all vertices in the sequence $\mathcal{H}$ which are after node $i$ and until node $j$, and returns to the depot $n+1$. Therefore, the cost of each arc is the sum of the hiring cost, the expected travel time cost, and the expected overtime cost of the route it represents. Since the graph considers all possible trips that can be obtained following the sequence of the giant tour, by finding the shortest path between $0$ and the last node in the sequence in $\mathcal{G}^{s}$ --which can be performed efficiently using Bellman-Ford's algorithm-- we find a set of optimal routes (subject to the order defined by $\mathcal{H}$). Interestingly, the time complexity of this procedure is polynomial in the input size. Figure \ref{fig:split} illustrates the split mechanism with the approximation algorithm. Figure  \ref{fig:split}(a) shows the minimum spanning tree, Figure \ref{fig:split}(b) depicts the giant tour, Figure \ref{fig:split}(c) presents the graph of possible trips where the dashed arcs denote the shortest path, and Figure \ref{fig:split}(d) exhibits the routes found.

\begin{figure}
    \centering
    \begin{subfigure}[b]{0.475\textwidth}
        \centering
        \scalebox{0.7}{
    \begin{tikzpicture}
        \node[od] (depot) at (1,2) {$0$};
        \node[intermediate](1) at (-4,3) {$2$};
        \node[intermediate](2) at (-2,4) {$3$}
            edge[-, >=stealth, line width=0.3mm, black]    (depot);
        \node[intermediate](3) at (4,4) {$5$}
            edge[-, >=stealth, line width=0.3mm, black]    (2);
        \node[intermediate](4) at (3,6) {$4$}
            edge[-, >=stealth, line width=0.3mm, black]    (2);
        \node[intermediate](5) at (-2,1) {$1$}
            edge[-, >=stealth, line width=0.3mm, black]    (depot)
            edge[-, >=stealth, line width=0.3mm, black]    (1);
        \draw [<->,white, >=stealth, line width=0.3mm] (5) to [out=350,in=250] (3);
        \draw [<->,white, >=stealth, line width=0.3mm] (1) to [out=50,in=170] (4);
    \end{tikzpicture}
    }
    \caption[]%
    {{\small Minimum spanning tree $\mathcal{T}$}}    
        \label{fig:mean and std of net14}
    \end{subfigure}
    \hfill
    \begin{subfigure}[b]{0.475\textwidth}  
        \centering 
       \scalebox{0.7}{
        \begin{tikzpicture}
            \node[od] (depot) at (1,2) {$0$};
            \node[intermediate](1) at (-4,3) {$2$}
                edge[-,dashed, >=stealth, line width=0.3mm]    (depot);
            \node[intermediate](2) at (-2,4) {$3$}
                edge[<-, >=stealth, line width=0.3mm]    (1)
                edge[-,dashed, >=stealth, line width=0.3mm]    (depot);
            \node[intermediate](3) at (4,4) {$5$}
                edge[->, >=stealth, line width=0.3mm]    (depot);
            \node[intermediate](4) at (3,6) {$4$}
                edge[<-, >=stealth, line width=0.3mm]    (2)
                edge[->, >=stealth, line width=0.3mm]    (3)
                edge[-,dashed, >=stealth, line width=0.3mm]    (depot);
            \node[intermediate](5) at (-2,1) {$1$}
                edge[<-, >=stealth, line width=0.3mm]    (depot)
                edge[->, >=stealth, line width=0.3mm]    (1);
            \draw [->,white, >=stealth, line width=0.3mm] (5) to [out=345,in=240] (depot);
        \end{tikzpicture}
        }
        \vspace*{1.5mm}
        \caption[]%
        {{\small Giant tour $\mathcal{H}$}}    
        \label{fig:mean and std of net24}
    \end{subfigure}
    \vskip\baselineskip
    \begin{subfigure}[b]{0.475\textwidth}   
        \centering 
        \scalebox{0.7}{
        \begin{tikzpicture}
            \node[od] (depot) at (0,0) {$0$};
            \node[intermediate](1) at (2,0) {$1$}
                edge[<-, >=stealth, line width=0.3mm]    (depot);
            \node[intermediate](2) at (4,0) {$2$}
                edge[<-, >=stealth, line width=0.3mm]    (1);
            \node[intermediate](3) at (6,0) {$3$}
                edge[<-, >=stealth, line width=0.3mm]    (2);
            \node[intermediate](4) at (8,0) {$4$}
                edge[<-, >=stealth, line width=0.3mm]    (3);
            \node[intermediate](5) at (10,0) {$5$}
                edge[<-, >=stealth, line width=0.3mm]    (4);
            
            \draw [->,black, >=stealth, line width=0.3mm] (depot) to [out=35,in=145] (2);
            \draw [->,dashed, black, >=stealth, line width=0.3mm] (depot) to [out=35,in=145] (3);
            \draw [->,black, >=stealth, line width=0.3mm] (depot) to [out=35,in=145] (4);
            \draw [->,black, >=stealth, line width=0.3mm] (depot) to [out=35,in=145] (5);
            
            \draw [->,black, >=stealth, line width=0.3mm] (1) to [out=325,in=235] (3);
            \draw [->,black, >=stealth, line width=0.3mm] (1) to [out=325,in=235] (4);
            \draw [->,black, >=stealth, line width=0.3mm] (1) to [out=325,in=235] (5);
            
            \draw [->,black, >=stealth, line width=0.3mm] (2) to [out=35,in=145] (4);
            \draw [->,black, >=stealth, line width=0.3mm] (2) to [out=35,in=145] (5);
            
            \draw [->,black, dashed,  >=stealth, line width=0.3mm] (3) to [out=325,in=235] (5);
            
        \end{tikzpicture}
        }
        \vspace*{1.5mm}
        \caption[]%
        {{\small Graph of possible trips $\mathcal{G}^{s}$}}    
        \label{fig:mean and std of net34}
    \end{subfigure}
    \hfill
    \begin{subfigure}[b]{0.475\textwidth}   
        \centering 
        \scalebox{0.7}{
        \begin{tikzpicture}
            \node[od] (depot) at (1,2) {$0$};
            \node[intermediate](1) at (-4,3) {$2$};
            \node[intermediate](2) at (-2,4) {$3$}
                edge[<-, >=stealth, line width=0.5mm, black]    (1)
                edge[->, >=stealth, line width=0.5mm, black]    (depot);
            \node[intermediate](3) at (4,4) {$5$}
                edge[->, >=stealth, line width=0.5mm,black]    (depot);
            \node[intermediate](4) at (3,6) {$4$}
                edge[->, >=stealth, line width=0.5mm, black]    (3)
                edge[<-, >=stealth, line width=0.5mm, black]    (depot);
            \node[intermediate](5) at (-2,1) {$1$}
                edge[<-, >=stealth, line width=0.5mm, black]    (depot)
                edge[->, >=stealth, line width=0.5mm, black]    (1);
                \draw [->,white, >=stealth, line width=0.3mm] (5) to [out=345,in=240] (depot);
        \end{tikzpicture}
        }
        \vspace*{1.5mm}
        \caption[]%
        {{\small Resulting routes}}    
    \end{subfigure}
        \vspace*{2mm}
        \caption[]
        {\small Illustrative example of the split mechanism with the approximation algorithm} 
        \label{fig:split}
\end{figure}
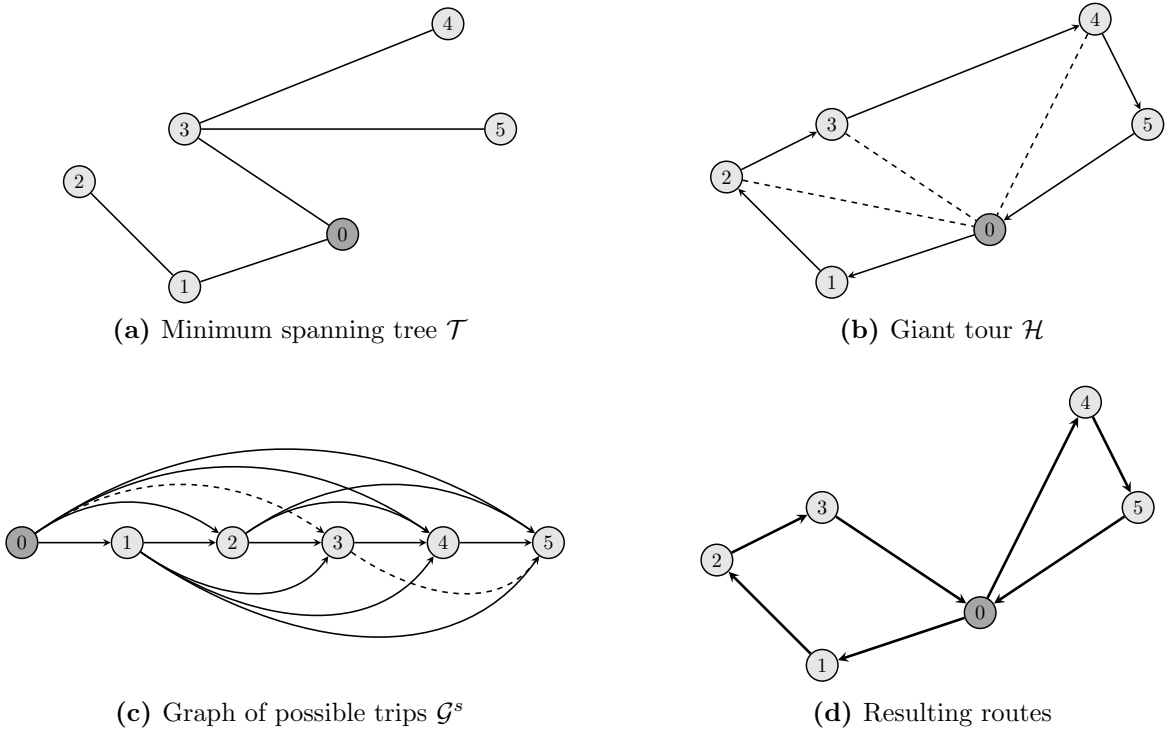

\subsection{Solving the pricing problem} \label{ss:pricing}

We use an integer programming formulation to solve the Elementary Shortest Path Problem. Since the costs $r_{ij} \in \mathbb{R}$ might induce negative cycles, the system of inequalities \eqref{balance}--\eqref{nature-pricing} becomes insufficient to guarantee the \emph{elementarity} of the path (\cite{taccari2016integer}). Thus, we must incorporate additional constraints to the integer program \eqref{FO_pricing}--\eqref{nature-pricing} to prevent \emph{subtours}. In particular, we adopt a \emph{cutting-plane} like approach in which we dynamically add \emph{generalized cutset inequalities} (GCS) to the formulation (\cite{taccari2016integer}). The GCS are defined as

\begin{align}
    \sum_{(i,j) \in \mathcal{A}(\mathcal{S})}b_{ijr} \leq \sum_{i\in \mathcal{S}\setminus \{u\}}\sum_{j\mid (i,j)\in \mathcal{A}}b_{ijr}, \quad \quad \quad \forall u \in \mathcal{S}, \mathcal{S}\subseteq \mathcal{V} \label{GCS}
\end{align}

where $\mathcal{S}$ denotes a subtour and $\mathcal{A}(\mathcal{S})$ is the set of arcs with both ends in $\mathcal{S}$. Algorithm \ref{IP-subproblem} presents the integer programming-based solution approach to solve the subproblem. We iteratively solve the integer program \eqref{FO_pricing}--\eqref{nature-pricing} and add the GCS when necessary. Once we find a solution without subtours, we return the corresponding elementary path. We make the following three observations. First, we tested other families of cuts (e.g., the \emph{Dantzig-Fulkerson-Johnson}) and the \emph{Miller-Tucker-Zemlin} model, but the GCS proved to be the most effective. Second, in an intermediate iteration of the \textbf{while} loop of Algorithm \ref{IP-subproblem}, the solution to the integer program might have an underlying elementary path. If that path has a negative reduced cost, we stop the execution and return it. Third, an integer-programming based solution method probably is less effective than a specialized algorithm, but it can be easily implemented in mathematical optimization software.

{\centering
\begin{algorithm}[ht]
\footnotesize
\DontPrintSemicolon
  \KwInput{$\mathcal{G}= (\mathcal{V},\mathcal{A})$, directed graph; $r_{ij}$, costs.}
  \KwOutput{$\mathcal{P}$, elementary path.} 
  \While{\texttt{true}}{
  \text{solve integer program \eqref{FO_pricing}--\eqref{nature-pricing}} \\
  \text{let $\mathcal{F}$ be the set of subtours in the solution found}\\
  \If{$|\mathcal{F}| = 0$}{
    let $\mathcal{P}$ be the elementary path found by the integer program\\
    \textbf{break}
  }\Else{
   \tcc{we add the GCS to break subtours}
  \For{$\mathcal{S} \in \mathcal{F}$}{
    add to the formulation constraints \eqref{GCS} for subtour $\mathcal{S}$
  }
  
  }
  
  }
  \Return{$\mathcal{P}$}
  \caption{Integer programming-based method for the pricing problem} 
  \label{IP-subproblem}
\end{algorithm}}

\rev{To accelerate the column generation algorithm, we also propose a heuristic method to solve the pricing problem. According to \textcite{desaulniers2002accelerating}, two of the most prominent acceleration strategies for column generation-based algorithms are: (i) adding a \emph{pool} of columns each iteration; and (ii) solving the pricing problem heuristically (i.e., not to optimality). To that end, we use the route-first cluster-second principle --as in Section \ref{ss: initial-solution}-- to solve the Elementary Shortest Path Problem. Specifically, we construct a giant tour $\mathcal{H}$ and a graph of possible trips $\mathcal{G}^{s}$ as before, except that using the reduced cost $r_{ij}$ to compute the weight of the arcs and considering only the nodes $i \in \mathcal{N}$ such that $\pi_{i}>0$. Thus, each time we execute the split, we obtain a set of feasible routes, and we keep those with a negative reduced cost. We remark the following. First, the costs $r_{ij}$ do not necessarily satisfy the triangle inequality, and therefore, the approximation ratio of Algorithm \ref{TSP-Approx} is no longer guaranteed. Second, in the graph of possible trips, the arcs' weight may be negative, but there cannot be negative cycles meaning the shortest path problem can be efficiently solved.} 

\subsection{Column generation-based heuristic} \label{ss: cg-heuristic}
The column generation algorithm described could be embedded in a branch-and-price scheme to find an optimal integer solution for the SAR problem. However, in order to find solutions in a reasonable time for on-time real-world applications, we propose a column generation-based heuristic described in Algorithm \ref{CG-heuristic}. The heuristic method finds new, promising routes solving the pricing problem until either the running time exceeds a predefined maximum running time $t_{\max}$ or no paths with negative reduced cost are found. Then, to find an integer solution, we solve the set covering model imposing the integrality conditions with the generated routes. Finally, the heuristic returns $\Omega^{*}$, the set of routes used in the solution for the SAR problem.

\rev{The proposed ideas can be embedded as a building block in different algorithms. For the purpose of the present study, we propose two algorithm configurations: (i) the Exact Method (EM), in which the pricing problem is solved exactly using the mathematical programming-based method; and (ii) the Heuristic Method (HM), in which we use the route-first cluster-second split mechanism to solve the subproblem. Since the EM solves the pricing problem exactly, its solution to the LP relaxation of the set covering model \eqref{11}--\eqref{13} is a dual (i.e., lower) bound on the optimal objective and can be used to compute an optimality gap. Regarding the parameter $t_{\max}$, note that it allows decision-makers to balance between the computational time and the quality of the solution. Note too that $t_{\max}=0$ implies solving the SAR problem with the initial solution.} 

{\centering
\begin{algorithm}[ht]
\footnotesize
\DontPrintSemicolon
  \KwInput{$\mathcal{G} = (\mathcal{V},\mathcal{A})$, directed graph; $t_{\max}$, maximum running time.}
  \KwOutput{$\Omega^{*}$, solution for the SAR problem.} 
  find an initial feasible solution \tcc*[r]{\emph{see \S \ref{ss: initial-solution}}}
  $\Omega \text{\textquotesingle} \leftarrow$ routes corresponding to the initial solution\\
  \While{running time $\leq t_{\max}$}{
  solve the LP relaxation of the set covering model \eqref{11}--\eqref{13}  \tcc*[r]{\emph{see \S \ref{ss:set covering}}}
  $\mathcal{P}\leftarrow$ path found by solving the pricing problem \tcc*[r]{\emph{see \S \ref{ss:pricing}}}
  \If{$\mathcal{P}$ has a negative reduced cost}{
  $\Omega \text{\textquotesingle} \leftarrow \Omega \text{\textquotesingle} \cup \mathcal{P}$
  }\Else{
    \textbf{break}
  }
  }
  let $\Omega^{*}$ be the optimal solution of the set covering model with integrality conditions \\
  \Return{$\Omega^{*}$}
  \caption{Column generation-based heuristic} 
  \label{CG-heuristic}
\end{algorithm}}

\section{Simulation-based method for the appointment scheduling problem} \label{s: appointment}
In this section, we tackle the appointment scheduling problem using a simulation-based method that captures the risk profile of decision-makers (\S \ref{ss:simulation-model}). Before presenting the simulation model, we discuss the characterization of stochastic travel time, service time, and cancellation (\S \ref{ss:stochastic-times}).

\subsection{Stochastic travel time, service time, and cancellation} \label{ss:stochastic-times}
Choosing the distribution of stochastic parameters correctly becomes critical since the results of the simulation model can be heavily determined by the assumptions made about those distributions. Therefore, we discuss the characterization of stochastic travel time, service time, and cancellation. For that purpose, we assume there is historical data from which we can compute the sample mean travel time $\hat{t}_{ij}$ between each pair of locations $i$ and $j$, the sample mean service time $\hat{s}$, and an estimated probability of cancellation $\hat{p}_{i}$ for each customer $i \in \mathcal{N}$.

First, we address stochastic travel times. For each arc $(i,j) \in \mathcal{A}$, we have a random variable $t_{ij}$ representing the traveling time from node $i$ to node $j$. Characterizing travel times on real-world transportation networks has been a major stream of research (e.g., \cite{tacs2013vehicle, gomez2016modeling, Mahmassani2012,Prakash2017,sun2011estimation}). For the purpose of this study, we use a simple yet robust method to obtain the distribution of travel times along the lines of \textcite{Prakash2017}. To estimate travel time variance $\hat{v}_{ij}$, we use a linear relation between the sampled mean $\hat{t}_{ij}$ and the standard deviation $\sqrt{\hat{v}_{ij}}$ of travel time. This relation measured in minutes per mile is given by the following equation (\cite{Mahmassani2012}) 
\begin{align}
\sqrt{\hat{v}_{ij}}= -0.4736 + 0.9936 \; \hat{t}_{ij} \nonumber
\end{align}

Moreover, \textcite{Prakash2017} claim that real-world travel times follow a lognormal distribution. In this light, for each arc $(i,j) \in \mathcal{A}$ we define 
\begin{align*}
    t_{ij} \sim \text{Lognormal}(\mu_{ij}, \sigma_{ij})
\end{align*}

Using the method of moments, the relation between the mean $\hat{t}_{ij}$ and variance $\hat{v}_{ij}$ of the travel time and the parameters of the lognormal distribution follows (\cite{Yamin2020, Prakash2018})

\begin{align}
\hat{\mu}_{ij} = \ln\left(\frac{\hat{t}_{ij}^2}{\sqrt{\hat{v}_{ij}+\hat{t}_{ij}^{2}}}\right) \quad \text{and} \quad \hat{\sigma}_{ij} = \sqrt{\ln\left(1+\frac{\hat{v}_{ij}}{\hat{t}_{ij}^{2}}\right)} \nonumber
\end{align}

Second, we consider stochastic service times. Since the exponential distribution is often used to represent service times (\cite{gomez2016modeling}), for each $i \in \mathcal{N}$, we have that
\begin{align*}
    s_{i} \sim \text{Exponential}(\lambda)
\end{align*}
Again, we use the method of moments to establish the relationship between the sample mean $\hat{s}$ and the parameter of the exponential distribution $\lambda$. We obtain the following equation
\begin{align*}
    \hat{\lambda} = \frac{1}{\hat{s}}
\end{align*}

Lastly, to capture the customers' cancellations during the day of service, we use a Bernoulli$(\gamma_{i})$ distribution. Given an estimated (or predicted) probability of cancellation $\hat{p}_{i}$ of customer $i$, by the method of moments, we have that $\hat{\gamma}_{i} = \hat{p}_{i}$. 

\rev{Using distributions that correctly model the stochastic parameters enhances our solution scheme as it becomes more robust and reliable. Nevertheless, the proposed ideas are not limited to these distributions and can be flexibly adapted to other families. To that end, Appendix B presents an overview of the method of moments in order to estimate the parameters of a given probability distribution. Also, we remark that one can find a well-suited distribution to a parameter by performing a goodness-of-fit test on the historical data set.}

\subsection{Simulation model} \label{ss:simulation-model}

To solve the H-SARA problem, we still need to schedule a visiting time for each customer; that is, we need to solve an appointment scheduling problem. Since we have previously solved the SAR problem with the column generation-based heuristic (see \S \ref{ss: cg-heuristic}), we already have the set of optimal routes $\Omega^{*}$. For each route $r\in \Omega^{*}$, we must establish the visiting time $w_{i}$ for each customer $i$ visited by that route. Therefore, we need to solve the appointment scheduling problem for each route independently. For the sake of clarity, we explain the simulation model for a given route $r$.

Let the random variable $\psi_{i}$ denote the arrival time to customer $i$ when visited by route $r$ and let $p_{\psi_{i}}(\cdot)$ be its probability distribution. Since $p_{\psi_{i}}(\cdot)$ can be difficult to find theoretically, we use a simulation-driven approach. More precisely, we conduct a Monte-Carlo simulation experiment, compute the arrival time to each node in the route, and obtain the \emph{empirical} distribution of $\psi_{i}$. \rev{We remark that the simulation model does account for the customers' cancellations. In particular, if the result of the Bernoulli trial is a success for a given customer --meaning that the customer cancels its appointment-- the service teams go directly to the location of the next customer on the route. Note that, in this way, appointments are scheduled early when the previous customer on the route has a high chance of canceling. Thus, when a cancellation occurs, the vehicles' idle time is minimized.}

We use $p_{\psi_{i}}(\cdot)$ to decide the appointment time for customer $i \in \mathcal{N}$. When doing so, we need to consider the costs for one time unit of earliness $c_{e}$ and one time unit of delay $c_{d}$. However, $c_{e}$ and $c_{d}$ can be difficult to estimate in real-world applications. Also, due to the high variance in urban transportation systems, it might be more relevant to consider a reliability measure rather than a cost. In this light, we define $\alpha \in [0,1]$ as the desired probability of on-time arrival (i.e., service team arriving before the scheduled time) to the customers' appointments. Correspondingly, $(1-\alpha)$ is the late-arrival probability referring to the probability of the service team arriving after the scheduled time. Given a value of $\alpha$ --which is specified by the decision-maker-- we simply set $w_{i}$ to the $\alpha$-percentile of the arrival time distribution $p_{\psi_{i}}(\cdot)$. On the one hand, if $\alpha$ is close to one --meaning the decision-maker is interested in maximizing the on-time arrival to the appointments-- the value of $w_{i}$ would lie on the \emph{right tail} of $\psi_{i}$, as Figure \ref{fig:empirical-distribution}(a) shows. On the other hand, if $\alpha$ is close to zero --meaning the decision-maker is interested in minimizing teams' idle time-- the value of $w_{i}$ would be lie on the \emph{left tail} of $\psi_{i}$, as Figure \ref{fig:empirical-distribution}(b) exhibits. Using a simulation-driven approach, our solution method captures decision-makers' risk profiles and ensures a reliable solution.

\begin{figure}
\centering
\begin{subfigure}{.5\textwidth}
  \centering
    \begin{tikzpicture}[
        declare function={gamma(\z)=
        2.506628274631*sqrt(1/\z)+ 0.20888568*(1/\z)^(1.5)+ 0.00870357*(1/\z)^(2.5)- (174.2106599*(1/\z)^(3.5))/25920- (715.6423511*(1/\z)^(4.5))/1244160)*exp((-ln(1/\z)-1)*\z;},
        declare function={gammapdf(\x,\k,\theta) = 1/(\theta^\k)*1/(gamma(\k))*\x^(\k-1)*exp(-\x/\theta);}, scale = 0.9
    ]
    \begin{axis}[
      no markers, domain=0:9, samples=100,
      axis lines=left, xlabel=$\psi_{i}$, ylabel=$p_{\psi_{i}}(\cdot)$,
      every axis y label/.style={at=(current axis.above origin),anchor=east},
      every axis x label/.style={at=(current axis.right of origin),anchor=north},
      height=5cm, width=9cm,
      xtick={9.0}, ytick=\empty,
      xticklabels={$w_{i}$},
      enlargelimits=false, clip=false, axis on top,
      grid = none
      ]
    
    \addplot [very thick,cyan!40!gray,domain=0:20] {gammapdf(x,2,2)};
    \addplot [fill=cyan!20, draw=none, domain=0:9.0] {gammapdf(x,2,2)} \closedcycle;
    \addplot [very thick, fill=red!40!gray, draw=none, domain=9.01:20] {gammapdf(x,2,2)} \closedcycle;
    \addplot +[color=black,mark=none, line width=1.5pt] coordinates {(9.01, 0) (9.01, 0.185)};
    \end{axis}
    \end{tikzpicture}
  \caption{Maximizing on-time arrival}
  \label{fig:case1}
\end{subfigure}%
\begin{subfigure}{.5\textwidth}
  \centering
    \begin{tikzpicture}[
        declare function={gamma(\z)=
        2.506628274631*sqrt(1/\z)+ 0.20888568*(1/\z)^(1.5)+ 0.00870357*(1/\z)^(2.5)- (174.2106599*(1/\z)^(3.5))/25920- (715.6423511*(1/\z)^(4.5))/1244160)*exp((-ln(1/\z)-1)*\z;},
        declare function={gammapdf(\x,\k,\theta) = 1/(\theta^\k)*1/(gamma(\k))*\x^(\k-1)*exp(-\x/\theta);}, scale = 0.9
    ]
    \begin{axis}[
      no markers, domain=0:9, samples=100,
      axis lines=left, xlabel=$\psi_{i}$, ylabel=$p_{\psi_{i}}(\cdot)$,
      every axis y label/.style={at=(current axis.above origin),anchor=east},
      every axis x label/.style={at=(current axis.right of origin),anchor=north},
      height=5cm, width=9cm,
      xtick={2.0}, ytick=\empty,
      xticklabels={$w_{i}$},
      enlargelimits=false, clip=false, axis on top,
      grid = none
      ]
    
    \addplot [very thick,cyan!40!gray,domain=0:20] {gammapdf(x,2,2)};
    \addplot [fill=cyan!20, draw=none, domain=0:2.0] {gammapdf(x,2,2)} \closedcycle;
    \addplot [very thick, fill=red!40!gray, draw=none, domain=2.01:20] {gammapdf(x,2,2)} \closedcycle;
    \addplot +[color=black,mark=none, line width=1.5pt] coordinates {(2, 0) (2, 0.185)};
    \end{axis}
    \end{tikzpicture}
  \caption{Minimizing teams' idle time}
  \label{fig:case2}
\end{subfigure}
\caption{Empirical arrival time distribution $p_{\psi_{i}}(\cdot)$. The blue area is the probability of on-time arrival. The red area is the probability of late-arrival.}
\label{fig:empirical-distribution}
\end{figure}
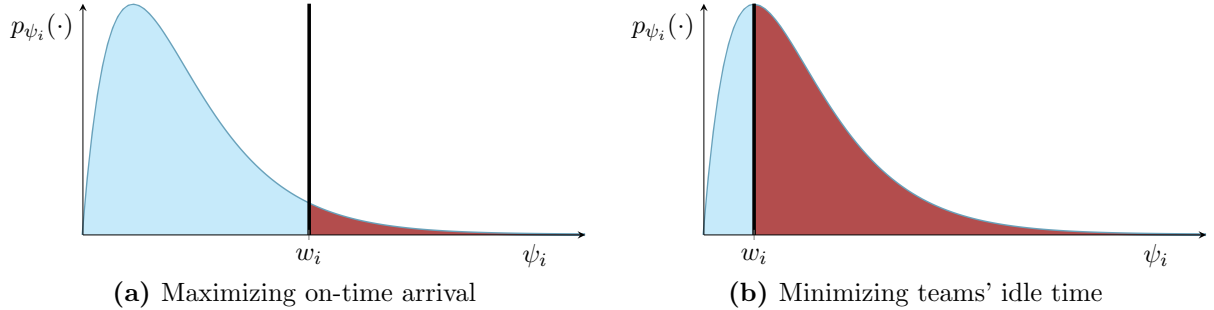

{\color{blue}
\section{Computational experiments} \label{s: computational-experiments}

In this section, we execute computational experiments to evaluate the performance of the proposed method for the H-SARA problem. In Section \ref{ss: runningTime}, we investigate the trade-off between computational time and quality of the solution. In Section \ref{ss: sensitivy}, we evaluate the impact of the cost parameters and the decision-makers' risk profiles on the solution. Finally, Section \ref{ss: AIMMS} presents a high-level decision support tool implemented in AIMMS. We coded all programs in Python and compiled them on a computer with an Intel Core i7-8665U CPU @ 2.11GHz with 16GB of memory. Also, we used Gurobi as mathematical optimization engine and \emph{NetworkX} for the graph algorithms. 

\subsection{Performance of the proposed method} \label{ss: runningTime}

In this set of experiments, we investigate the trade-off between the computational time and the quality of the solution provided by the column generation algorithm for the SAR problem. To do so, we compare two configurations of the algorithm: the Exact Method (EM) and the Heuristic Method (HM). To compare the gains of solving the pricing problem exactly and heuristically, we use performance metrics such as the speedup and the optimality gap. We also evaluate the quality of the Initial Solution (IS).

We generate the test instances as follows. We assume that the depot is at the origin node, and the customers are distributed uniformly in a square of edge 50 km with the origin at the center. The mean value of travel time between two locations is equal to the Euclidean distance since the mean travel speed is assumed to be 1 km/min. The mean customer service times are randomly generated from a uniform distribution in the interval $[30, 60]$. Furthermore, we set $L = 250$, $c_f = 100$, $c_{t} = 1$, and $c_{o} = 2$. For the simulation model, we consider 100 replicas and a probability of cancellation equal to 0.1 for each customer, and we set the value of $\alpha$ to $0.5$. We execute each configuration of the proposed method in 10 randomly generated instances varying the number of customers $n=10,20,30,40,50,200,500$, leading to a total of 100 runs (we did not execute the EM in test cases with more than 30 customers since the computational time exceeded three hours). Lastly, we set $t_{\max} = \infty$ which means that the route-finding procedure stops once no paths with negative reduced cost are found. To show the scalability of the proposed ideas, we execute the method with 1,000 and 3,000 customers, and solved the SAR problem with the IS. For this last experiment, we consider the costumers to be in a square of edge 400 km with the depot at the center.  

Table \ref{tab:assessmet} summarizes the main results. Column 1 shows the number of customers. Columns 2 through 7 present the average objective (Avg. \texttt{Obj}) and running time in seconds (Avg. \texttt{CPU}) of the IS, the HM, and the EM, respectively. Column 8 reports the average \texttt{speedup} achieved by the HM against the EM. Finally, columns 9 and 10 exhibit the average \texttt{gap} --calculated as in equation \eqref{gap}-- between the objective found by the IS and the HM with respect to the objective found by the EM. For the EM, we report the Avg. \texttt{Obj} of the LP relaxation of the set covering model (since it is a lower bound on the optimal objective) whereas, for the HM, we report the  Avg. \texttt{Obj} of the integer set covering problem solved with the generated routes.  

\begin{align}
    \texttt{gap}(z,z') = \frac{(z-z')}{z'} \times 100
    \label{gap}
\end{align}

\begin{table}[ht]
    \centering
    \scalebox{0.75}{
    \begin{tabular}{c|cc|cc|cc|c|cc} \hline \hline
        \multicolumn{1}{c|}{$n$} & \multicolumn{2}{c|}{IS} & \multicolumn{2}{c|}{HM} & \multicolumn{2}{c|}{EM} & \multicolumn{1}{c|}{Avg. \texttt{speedup}} & \multicolumn{2}{c}{Avg. \texttt{gap} (\%)}\\
        \hline
        \multirow{2}{*}{Customers} & \multirow{2}{*}{Avg. \texttt{Obj}} & \multirow{2}{*}{Avg. \texttt{CPU}} & \multirow{2}{*}{Avg. \texttt{Obj}} & \multirow{2}{*}{Avg. \texttt{CPU}} &
        \multirow{2}{*}{Avg. \texttt{Obj}} & \multirow{2}{*}{Avg. \texttt{CPU}} &
        \multirow{2}{*}{EM/HM} & \multirow{2}{*}{(IS,EM)} & \multirow{2}{*}{(HM,EM)}
        \\ & & & & & & & & & \\
        \hline
    10 & 511.32 & 0.00 & 505.32 & 0.03 & 467.93 & 4.91 & 176.03 & 9.27 & 8.05\\
    20 & 908.94 & 0.00 & 867.86 & 0.08 & 811.67 & 307.67 & 4,282.48 & 11.92 & 6.81\\
    30 & 1,306.51 & 0.01 & 1,283.05 & 0.29 & 1,174.36 & 7,441.67 & 33,304.38 & 11.25 & 9.26\\
    40 & 1,647.48 & 0.01 & 1,633.26 & 0.33 &  -  &  -  &  -  &  -  &  - \\
    50 & 1,999.88 & 0.04 & 1,982.81 & 0.66 &  -  &  -  &  -  &  -  &  - \\
    200 & 7,234.94 & 0.66 & 7,209.98 & 35.23 &  -  &  -  &  -  &  -  &  - \\
    500 & 17,456.96 & 9.18 & 17,377.76 & 1,709.62 &  -  &  -  &  -  &  -  &  - \\
    \hline
    1,000 & 123,303.80 & 111.34 & - & - & - & - & - & - & - \\
    3,000 & 324,215.30 & 3,001.84 & - & - & - & - & - & - & - \\
    \hline \hline
    \end{tabular}}
    \caption{Performance of the proposed method}
    \label{tab:assessmet}
\end{table}

From Table \ref{tab:assessmet} we highlight the following takeaways. First, the initial solution is found in less than a second (except for the instances with 500 customers) and is always within $12\%$ of the optimal objective. Second, we observe that the HM can be 33,000 times faster than the EM. Interestingly, the running time of the HM seems to increase linearly with the number of customers, indicating its capacity to scale and solve large-scale instances efficiently. Moreover, the optimality gap of the solution found by the HM ranges from $8.05\%$ to $9.26\%$, which is likely to be sufficient for real-world applications. Third, we see that the column generation procedure reduces the optimality gap with respect to the initial solution by approximately $1\%$ to $5\%$. We note that the running time of the simulation model was, on average, just 10\% of the total running time, meaning it is an inexpensive procedure. We conclude that the HM finds good solutions in a reasonable time, even for large-scale, real-world instances. However, the EM provides better solutions for long-term planning applications at the expense of a heavy computational burden. Finally, regarding the instances with $1,000$ and $3,000$ customers, we see that the method is able to solve very large test cases using the IS.

\subsection{Sensitivity analysis on the costs and the decision-makers profiles} \label{ss: sensitivy}

In this set of experiments, we perform a sensitivity analysis on the costs and the decision-makers profiles. In particular, we analyze the changes in the solution and the performance of the algorithm when varying the hiring cost, the unit travel time cost, the unit overtime cost, and the desired probability of on-time arrival. We generate the instances as described in Section \ref{ss: runningTime} but considering only 20, 30, and 40 customers and executing only the HM.

First, we investigate the impact of changing the cost-related parameters. To do so, we use three configurations: A (high hiring cost), B (high overtime cost), and C (high travel time cost). We execute the HM on each configuration 15 times, leading to a total of 135 runs. Figure \ref{f_costsSensitivity}(a) and Figure \ref{f_costsSensitivity}(b) present, respectively, the box-plots with the number of routes used and the running time for each configuration varying the number of customers.

\begin{figure}[ht]
  \centering
    \begin{subfigure}[b]{0.48\textwidth}
        \includegraphics[width=0.95\textwidth]{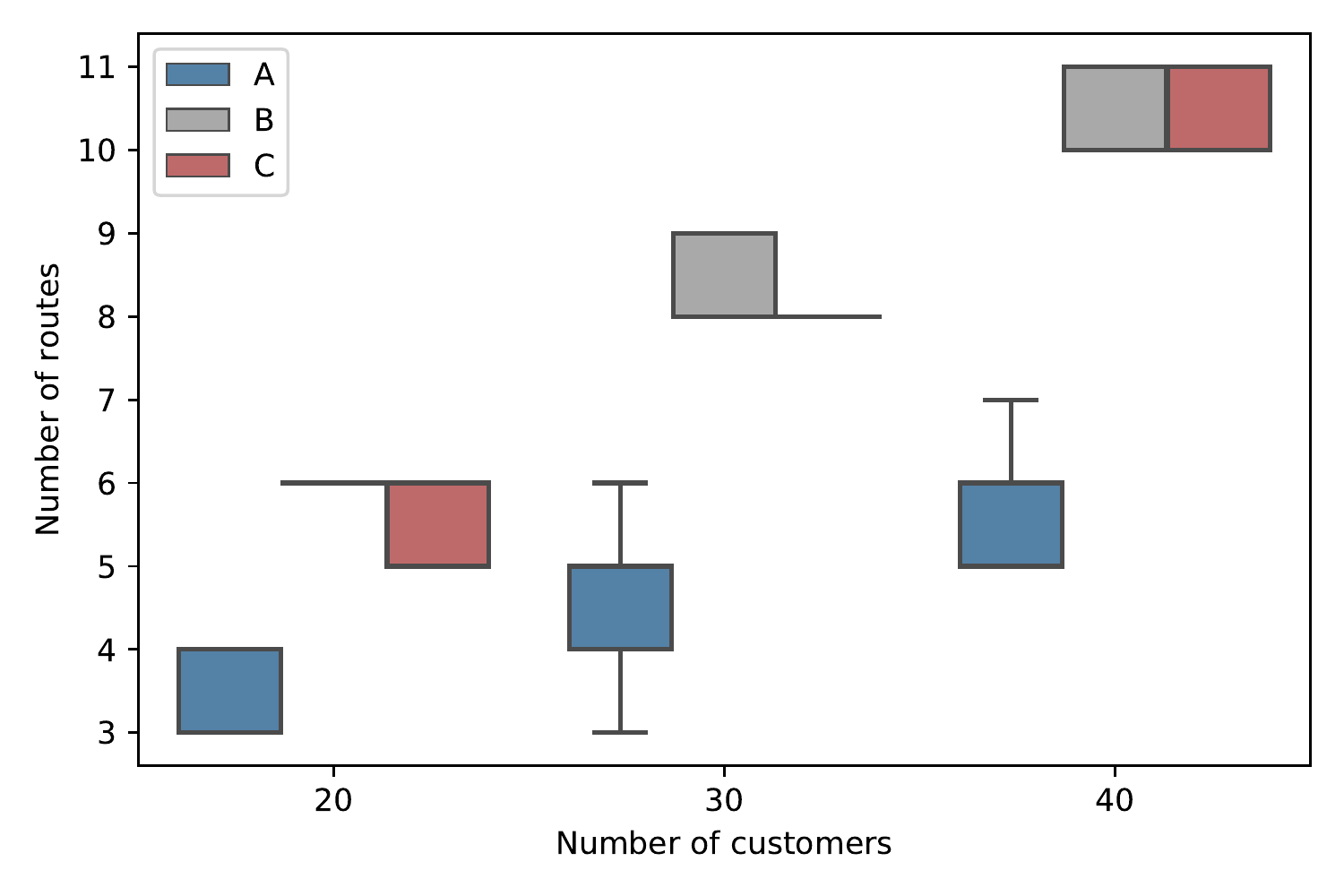}
        \caption{Number of routes}
    \end{subfigure}%
    \begin{subfigure}[b]{0.48\textwidth}
        \includegraphics[width=0.95\textwidth]{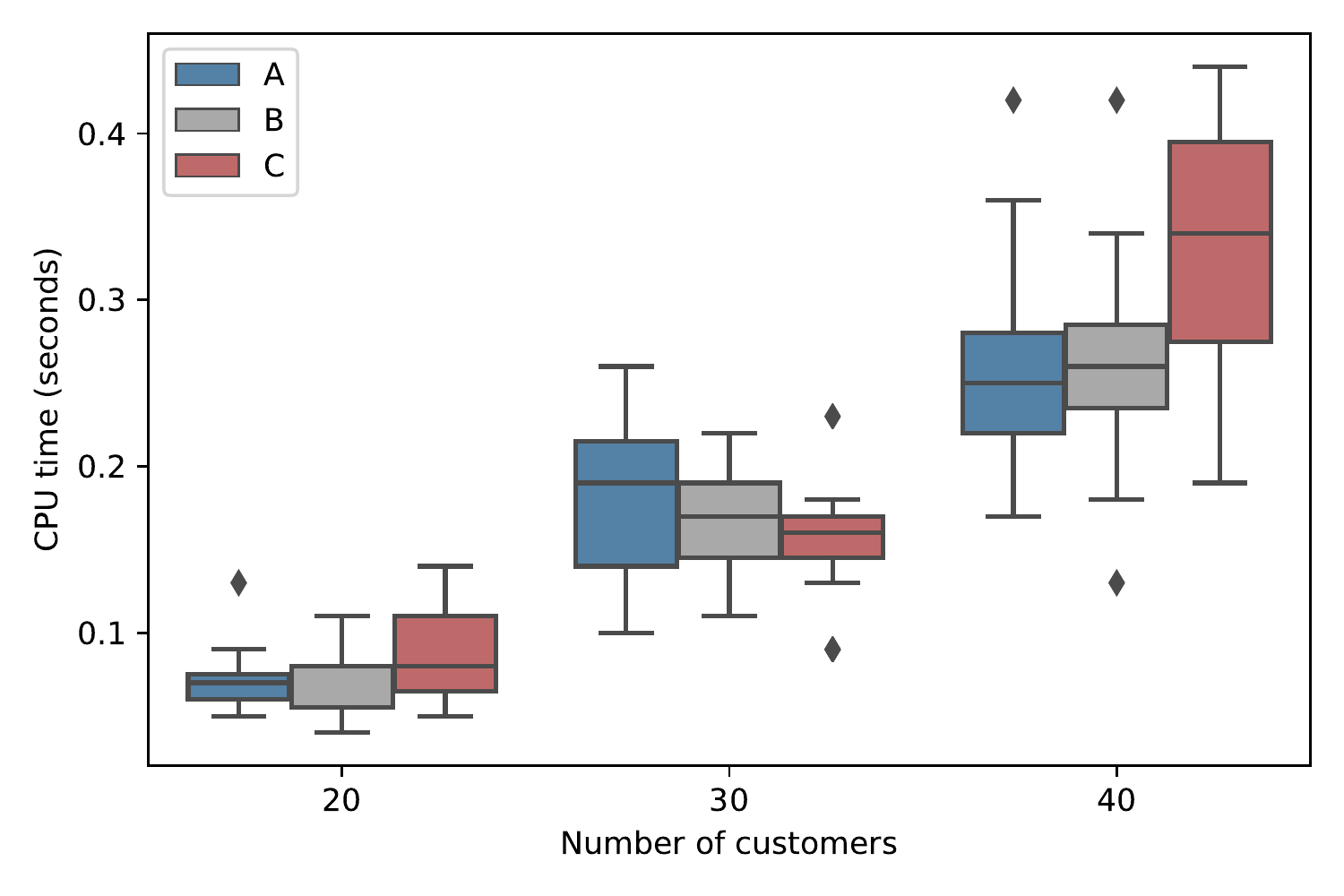}
        \caption{Running time}
    \end{subfigure}%
\caption{Sensitivity analysis on the costs}
\label{f_costsSensitivity}
\end{figure}

From Figure \ref{f_costsSensitivity}(a), we observe that, as expected, configuration A uses fewer routes than configurations B and C because of the high fixed cost of hiring a team. We also see that the number of routes does not vary much for a fixed configuration and number of customers. Regarding Figure \ref{f_costsSensitivity}(b), we note that there is not a common trend to conclude that the changes in the cost-related parameters influence the running time of the method. Furthermore, these results confirm those discussed in Section \ref{ss: runningTime} since the running time of the HM seems to scale linearly with the number of customers (regardless of the configuration). 

Second, we assess the effect of varying the desired probability of on-time arrival. For that purpose, we use three configurations: X (maximizing on-time arrival), Y (neutral), and Z (minimizing teams' idle time). Again, we execute the HM on each configuration 15 times, leading to a total of 135 runs. Figure \ref{f_alphaSensitivity}(a) shows the average difference in the scheduled return time to the depot among the three configurations. We refer to the scheduled return time to the depot as the sum of the latest appointment time in the route, the mean service time (of attending that last customer), and the expected travel time from that last customer to the depot. Figure \ref{f_alphaSensitivity}(b) presents the running time for each configuration varying the number of customers.

\begin{figure}[ht]
  \centering
    \begin{subfigure}[b]{0.48\textwidth}
        \includegraphics[width=0.95\textwidth]{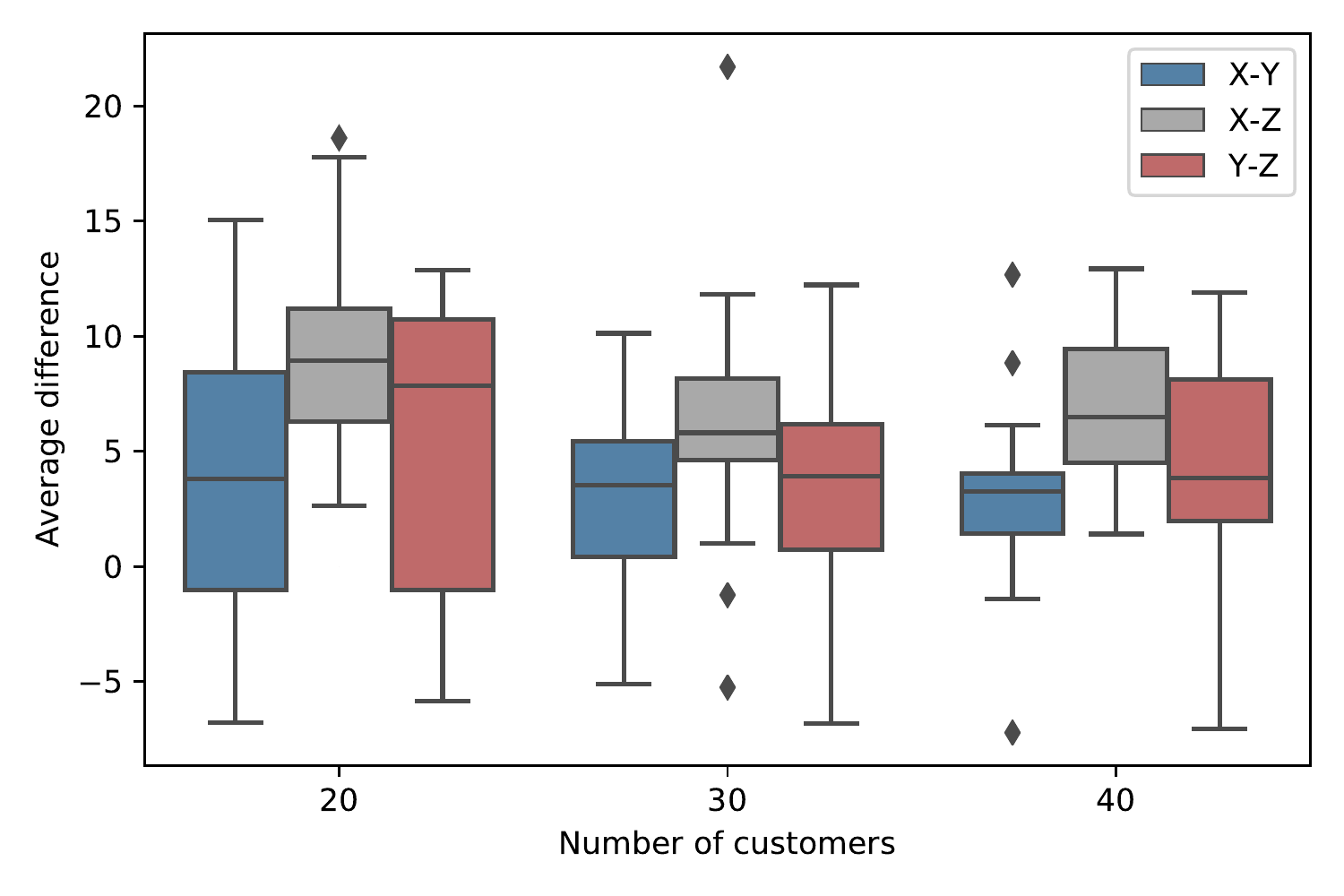}
        \caption{Difference in the return time to the depot}
    \end{subfigure}%
    \begin{subfigure}[b]{0.48\textwidth}
        \includegraphics[width=0.95\textwidth]{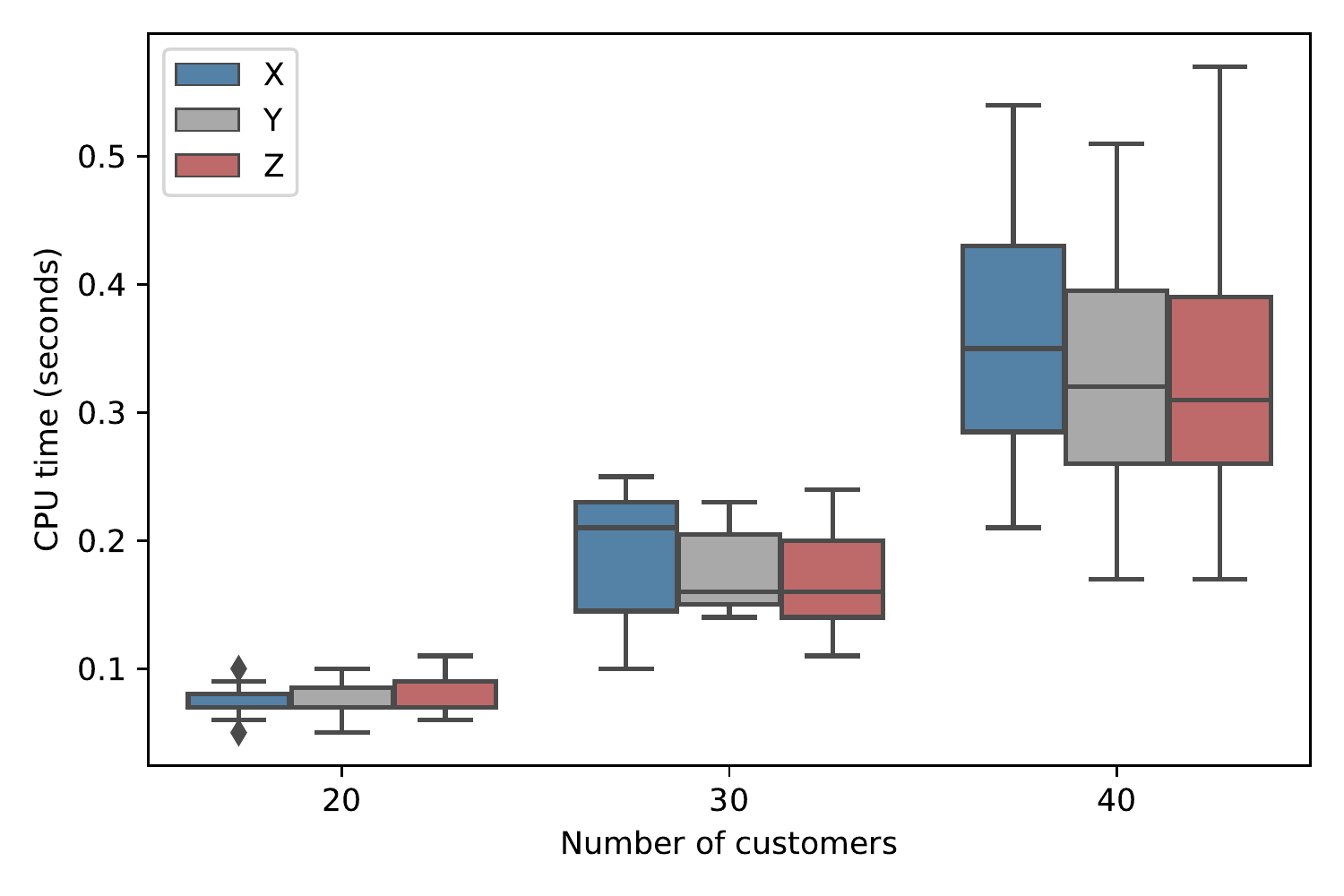}
        \caption{Running time}
    \end{subfigure}%
\caption{Sensitivity analysis on the decision-makers profiles}
\label{f_alphaSensitivity}
\end{figure}

From Figure \ref{f_alphaSensitivity}(a), we see that the average differences in the scheduled return time to the depot between each pair of configurations is frequently greater than zero. This is an expected result considering that configuration Z schedules appointments early whereas configuration X schedules appointments late; therefore, Z's scheduled return should be prior to that of X (with that of Y in between). Nevertheless, for each configuration, we conducted an independent Monte-Carlo simulation. For this reason, the difference between configurations is less than zero in some cases. As was foreseeable, the largest average difference is between configurations X and Z. We remark that the magnitude of the differences between configurations depends on the shape of the empirical distribution that ultimately depends on the assumptions made. Regarding Figure \ref{f_alphaSensitivity}(b), we note that there is no clear pattern, indicating that the running time of the method is not affected by the parameter $\alpha$, that is, by the decision-maker risk profiles.

\subsection{High-level decision support system in AIMMS} \label{ss: AIMMS}

We embedded our solution method into a user-friendly high-level decision support system (DSS) developed in AIMMS. This computerized program can be used in decision-making in the home service industry, allowing decision-makers to solve problems timely, improve their efficiency, and perform other critical tasks related to operations, planning, and management. Conveniently, the DSS is flexible and adaptable to accommodate changes in the environment and the user's decision-making. The application can be used by upper- and mid-level managers to analyze multiple outcomes based on optimization models and analytical techniques.

We discuss some of the essential features of the AIMMS-based application. First, as Figure \ref{fig:application} shows, the user interface of the DSS is intuitive, easy to use, and visually attractive to help managers interact with the tool. Second, the DSS provides a real-world geographical visualization of the solution, facilitating planning and operations-related tasks. Third, the information system presents a cost breakdown so that users gain insight into their operational costs. This feature allows decision-makers to decrease costs, a critical activity in a competitive market. Fourth, the DSS details information regarding the cost, travel time, and visiting customers of each route to help automate the managerial processes. Fifth, the application allows users to perform a what-if analysis on the costs, the desired on-time probability, and the number of customers serviced. This analysis reveals new approaches for the company to increase its competitiveness. Sixth, the DSS outlines a schedule for each customer, including its appointment time and route, to increase the control of the operation. Seventh, the users can specify a maximum running time to balance between the computational time and the quality of the solution. We highlight that this feature allows managers to find solutions within seconds for on-time applications or near-optimal solutions (at the expense of longer computational time) for long-term planning applications. Finally, the information can be uploaded to the application from an external file to facilitate the handling of large data sets. In short, the AIMMS-based DSS helps managers make critical decisions based on analytical methods improving the efficiency and competitiveness of their companies. 

\begin{figure}[ht]
\centering
\includegraphics[width=.90\linewidth]{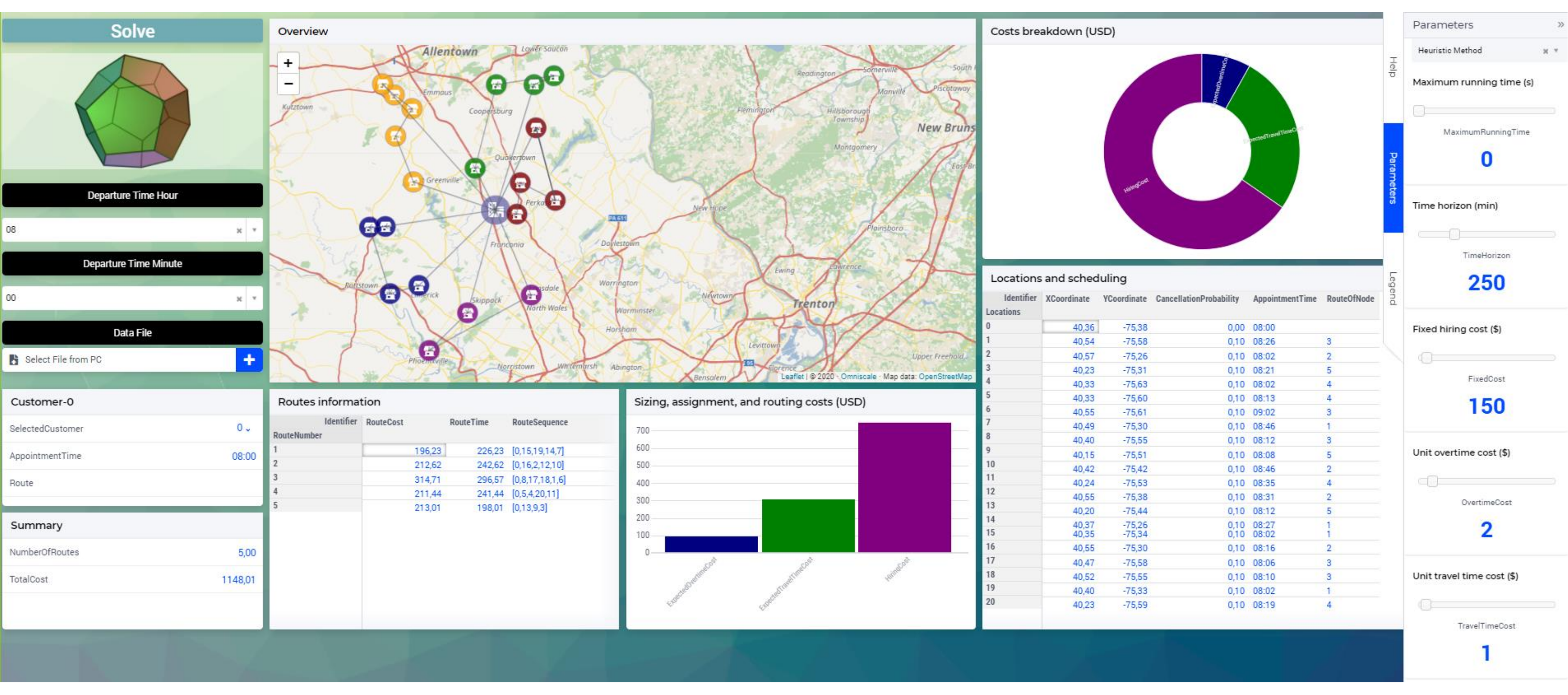}
\caption{Dashboard of the AIMMS-based decision support system}
\label{fig:application}
\end{figure}

}

\section{Concluding remarks} \label{s: conclusion}
In this study, we tackled the \emph{Home Service Assignment, Routing, and Appointment scheduling} (H-SARA) problem problem by proposing a two-stage solution scheme. In the first stage, we used a column generation-based heuristic to solve the sizing, assignment, and routing (SAR) problem. The column generation algorithm is enhanced by a high-quality initial solution found using the route-first cluster-second principle and a polynomial-time 2-approximation algorithm. In the second stage, we proposed a simulation-driven approach to solve the appointment scheduling problem that guarantees reliability in the solution found according to the decision-makers' profile. To ensure the suitability of the simulation model, we discussed the characterization of the stochastic parameters. The computational study concluded that the Heuristic Method (HM) finds good solutions in a reasonable time for real-world applications whereas, for long-term planning applications, the Exact Method (EM) can provide a better solution at the expense of longer computational running time. Furthermore, a sensitivity analysis showed that changing the parameters does not influence the running time of the method and that the solution scheme can adapt to several settings and profiles. Finally, the AIMMS-based decision support system (DSS) helps managers make critical decisions based on analytical methods improving the efficiency and competitiveness of their companies. Future research could embed the proposed ideas into a branch-and-price algorithm, test the method on real-world instances, or implement a specialized exact algorithm for the pricing problem.

{\color{blue}
\section*{Acknowledgements}
The authors would like to thank everyone involved in the 13th AIMMS-MOPTA Optimization Modeling Competition for their arduous work organizing this enriching event. Also, we thank the members of the Center for Optimization and Applied Probability (COPA) for the discussion that improved the article.
}

\singlespacing
\printbibliography

\appendix

\section{Notation table}\label{A:Notation}
 \begin{table}[H]
    \begin{center}
    \scalebox{0.70}{
    {\color{blue}
    \begin{tabular}{cl}
    \hline \hline
    Notation & Description 
    \\ \hline
    \multicolumn{2}{c}{\underline{Problem definition}} \\
    $\mathcal{G}=(\mathcal{V},\mathcal{A})$ & directed graph in which $\mathcal{V}$ is the set of nodes, $\mathcal{N} \subseteq \mathcal{V}$ is the set customers, and $\mathcal{A}$ is the set of arcs.\\
    $s_{i}$ & service time of customer $i \in \mathcal{N}$. \\
    $p_{i}$ & probability associated with customer $i\in \mathcal{N}$ canceling its appointment.\\
    $t_{ij}$ & travel time between $i \in \mathcal{V}$ and $j \in \mathcal{V}$. \\
    $\mathcal{M}$ & set of homogeneous service teams. \\
    $L$ & time limit. \\
    $c_{f}$ & fixed hiring cost. \\
    $c_{t}$ & cost for one unit of travel time. \\
    $c_{o}$ & cost for unit of overtime. \\
    $c_{e}$ & cost for one unit of earliness. \\
    $c_{d}$ & cost for one unit of delay. \\
    $x_{ij}^{k}$ & binary flow variable of vehicle $k \in \mathcal{M}$ along arc $(i,j) \in \mathcal{A}$. \\
    $w_{i}^{k}$ & expected start time of service of vehicle $k \in \mathcal{M}$ at node $i \in \mathcal{V}$. \\
    $\Delta^{k}$ & auxiliary overtime variable of vehicle $k \in \mathcal{M}$. \\
    
    \multicolumn{2}{c}{\underline{Solution method}} \\
    $\Omega$ & set of feasible routes satisfying constraints \eqref{3}--\eqref{5}. \\
    $c_{r}$ & cost of route $r\in \Omega$ including the hiring, travel time, and overtime costs. \\
    $a_{ir}$ & binary parameter indicating whether the route visits node $i\in \mathcal{N}$ or not.\\
    $b_{ir}$ & binary parameter indicating whether the route uses arc $(i,j) \in \mathcal{A}$ or not.\\
    $y_{r}$ & binary variable indicating if route $r\in \Omega$ is used in the solution or not.\\
    $\pi_{i}$ & dual variable of the covering constraint \eqref{12}. \\ 
    $r_{ij}$ & reduced cost. Refer to equation \eqref{reducedCost}.\\
    $\tilde{t}_{ij}$ & modified travel time. Refer to equation \eqref{time}.\\
    $\mathcal{G}^{t} = (\mathcal{V}^{t}, \mathcal{A}^{t})$ & undirected graph to compute the minimum spanning tree. \\
    $\mathcal{H}$ & Hamiltonian cycle (giant tour). \\
    $\mathcal{G}^{s} = (\mathcal{V}^{s}, \mathcal{A}^{s})$ & directed acyclic graph of possible trips. \\
    $t_{\max}$ & maximum running time. \\
    $\hat{t}_{ij}, \hat{v}_{ij}$ & sampled mean and variance of travel time. \\
    $\mu_{ij}, \sigma_{ij}$ & parameters of the lognormal distribution. \\
    $\hat{s}$ & sampled mean service time. \\
    $\lambda$ & parameter of the exponential distribution. \\
    $\hat{p}_{i}$ & estimated probability of cancellation of customer $i \in \mathcal{N}$. \\
    $\gamma_{i}$ & parameter of the Bernoulli distribution. \\
    $\psi_{i}$ & random arrival time to customer $i \in \mathcal{N}$ when visited by route $r \in \Omega$.\\
    $p_{\psi_{i}}(\cdot)$ & probability distribution of $\psi_{i}$. \\
    $\alpha$ & desired probability of on-time arrival. \\
    \hline \hline
    \end{tabular}
    }}
    \end{center}
    \caption{Notation table}
    \label{NotationTable}
    \end{table}%

\section{Method of moments} \label{A:Moments}

\rev{In this section, we explain the procedure to use the method of moments in order to estimate the unknown parameters $\theta_{1}, \theta_{2}, \ldots, \theta_{k}$ that characterize the probability distribution $f_{X}(x;\theta)$ of random variable $X$. Suppose that the first $k$ moments of the distribution can be expressed as
\begin{align*}
    \mu_{1} & \equiv E[X] = g_{1}(\theta_{1}, \theta_{2}, \ldots, \theta_{k}) \\
    \mu_{2} & \equiv E[X^{2}] = g_{2}(\theta_{1}, \theta_{2}, \ldots, \theta_{k}) \\
    & \; \; \vdots \nonumber\\
    \mu_{k} & \equiv E[X^{k}] = g_{k}(\theta_{1}, \theta_{2}, \ldots, \theta_{k})
\end{align*}
Then, using a sample of size $n$, we obtain the values $x_{1}, \ldots, x_{n}$. For $j = 1, \ldots, k$, we estimate $\mu_{j}$ using the $j$-th sample moment given by
\begin{align*}
    \hat{\mu}_{j} = \frac{1}{n}\sum_{i = 1}^{n}x_{i}^{j}
\end{align*}
Therefore, the estimators found by the method of moments --denoted as $\hat{\theta}_{1}, \ldots, \hat{\theta}_{k}$-- are defined by the solution (provided it exists) of the following system of equations
\begin{align*}
    \mu_{1} & = g_{1}(\hat{\theta}_{1}, \hat{\theta}_{2}, \ldots, \hat{\theta}_{k}) \\
    \mu_{2} & = g_{2}(\hat{\theta}_{1}, \hat{\theta}_{2}, \ldots, \hat{\theta}_{k}) \\
    & \; \; \vdots \nonumber\\
    \mu_{k} & = g_{k}(\hat{\theta}_{1}, \hat{\theta}_{2}, \ldots, \hat{\theta}_{k}).
\end{align*}
}
\end{document}